\documentclass[11pt]{extarticle}

\usepackage{amsmath,amsfonts,amssymb,latexsym,amsthm}
\usepackage{graphicx}
\usepackage{color}
\usepackage{eulervm}

\newtheorem{theorem}{Theorem}[section]
\newtheorem{prop}{Proposition}[section]
\theoremstyle{definition}
\newtheorem{defn}{Definition}[section]
\theoremstyle{remark}
\newtheorem{remark}{Remark}[section]

\numberwithin{equation}{section}
\allowdisplaybreaks

\makeatletter
\let\@fnsymbol\@arabic
\makeatother

\title{On a class of Time-fractional Continuous-state Branching Processes}

\author{Luisa Andreis\thanks{Dipartimento di Matematica ``T. L. Civita'', Universit\`a degli Studi di Padova, via Trieste 63, 35121 Padova (Italy);
		e-mail addresses: andreis@math.unipd.it} 
		\and Federico Polito\thanks{Dipartimento di Matematica ``G.\ Peano", Universit\`a degli Studi di Torino, via Carlo Alberto 10,
		10123 Torino (Italy); e-mail address: $\{$federico.polito, laura.sacerdote$\}$@unito.it}  \and Laura Sacerdote\footnotemark[2]}

\begin{document}

	\maketitle

	\begin{abstract}

		\noindent 
		
		We propose a class of non-Markov population models with continuous or discrete state space
		via a limiting procedure involving sequences of rescaled and randomly time-changed Galton--Watson processes.
		The class includes as specific cases the classical continuous-state branching processes and Markov branching processes.
		Several results such as the expressions of moments and the branching inequality governing the evolution of the
		process are presented and commented. The generalized Feller branching
		diffusion and the fractional Yule process are analyzed in detail as special cases of the general model.  
		
		\vspace{0.1cm}

		
	\end{abstract}

	\section{Introduction} 
		
		Since the seminal paper of Galton-Watson \cite{watson1875probability}, branching structures are subject to intensive theoretical
		and applied researches. The most studied applications of branching phenomena concern population growth models.
		In this context, in 1958, M.\ Ji{\v{r}}ina \cite{Jir58} introduced the so-called continuous-state branching processes (shortly CSBPs)
		that represent a general class of linear branching processes in which jumps of any finite size and a continuous state space are permitted
		(see also \cite{Li10} and the references therein). The original definition of CSBPs is very similar to that of
		L\'evy processes (with which they are linked by means of a random time change, the Lamperti transform). However, an
		alternative definition, dating back to the work of J.\ Lamperti \cite{Lam67} considers CSBPs as limit processes of sequences
		of rescaled Galton--Watson processes (GWPs in the following) or Markov branching processes
		(see also \cite{AlSh83,Gri74} for further references).
		Due to their simple definition, generalizations of the GWPs and CSBPs  have arisen in several directions, leading for example
		to the introduction of population-size-dependent GWPs and CSBPs~\cite{Kleb84, Li06},
		and controlled branching processes~\cite{SeZu74}, where the independence of individuals' reproduction is modified allowing
		dependence on the size of the current population. In this paper we aim to extend the definition of GWPs and CSBPs in
		a different direction. Indeed, the Markov property characterizing these processes, although is mathematically appealing,
		determines a limitation for their actual application; furthermore, non-Markov branching processes would present interesting
		mathematical properties that constitute a reason of study by itself. 
		Here we introduce a general class of non-Markov population models characterized by persistent memory and contructed by means
		of a limiting procedure on a sequence of suitably rescaled Galton--Watson processes time-changed by a specific random process.
		In order to clarify our approach, we briefly recall how time-changes play a fundamental role in the definition of models
		for anomalous diffusion. We will take inspiration from them.
		Roughly speaking, the basic framework is the following. Take a standard Brownian motion,
		say $\{B(t), t \ge 0\}$, and an independent stable subordinator $d=\{D(t), t\ge 0\}$,
		that is a spectrally positive increasing L\'evy process with stable unilateral probability density function.
		Define the inverse process to $D$ as
		\begin{equation*}
			\mathcal{E}(t):=\inf\{u>0 \colon D(u)>t\}, \qquad t \ge 0.
		\end{equation*}
		Then, the time-changed process $\{B(\mathcal{E}(t)), t \ge 0\}$ is a non-Markov process with continuous sample paths and exhibiting
		a sub-diffusive behaviour. Furthermore, if $\mathbf{P}(B(\mathcal{E}(t)) \in dx)/dx = l(x,t)$ is the marginal probability density
		function of the time-changed Brownian motion, then $l(x,t)$ solves the fractional PDE
		\begin{align*}
			\partial_t^\beta l(x,t) = \frac{1}{2} \frac{\partial^2}{\partial x^2}l(x,t), \qquad t \ge 0, \: x \in \mathbb{R}, \: \beta \in (0,1).
		\end{align*}
		The above operator acting on time is a non-local integro-differential operator called D\v{z}rba\v{s}jan--Caputo derivative
		(see Section \ref{back} for prerequisites and specific information) and $\beta$ is the stability parameter. The main consequence
		of the presence of the fractional derivative is that, due to non-locality, it furnishes the model with a long memory.
		
		Hence, in this paper we build via a limiting procedure and specific time-changes a large class of processes with branching structure
		also exhibiting non-locality and long memory. This is actually carried out in Section~\ref{violin}. Specific cases of interest
		being part of this class are,
		amongst others, the generalized Feller branching diffusion and the fractional Yule process.

		Due to the nature of the considered problem, the paper fits exactly in-between two classical topics of probability,
		namely population models (processes exhibiting a branching structure) and models for anomalous diffusion
		(frequently associated to fractional diffusion).

		The paper is organized as follows: in Section~\ref{back} we introduce the notation and recall the basic definitions and properties
		that we use in the sequel; in Section~\ref{Limite} we define the time-changed processes both in the discrete and the continuous
		setting and we prove the scaling limit; in Section~\ref{sec_properties} we focus on the time-changed CSBPs with the proof of some
		properties and some examples.

	\section{Backgrounds}\label{back}
		
		The aim of this section is to give a brief overview of the processes we are interested in. We recall the definition and some basic
		properties of GWPs and of CSBPs; in particular the branching property is of
		fundamental importance. Moreover, basic information on fractional calculus and fractional diffusion is also recalled.

		\subsection{From GWPs to CSBPs}

			GWPs are classical discrete-time branching processes, where each individual of a population reproduces
			independently and according to the same offspring distribution $p$, see \cite{AtNe72} for a complete introduction.
			Rigorously, given a probability measure $p$ on $\mathbb{N}$, a GWP $\{Z_n\}_{n\geq0}$ with offspring distribution $p$
			is the Markov chain such that, for all $n\geq 0$,
			\begin{equation*}
				Z_{n+1}\stackrel{d}{=}\sum_{i=1}^{Z_n}\xi_i,
			\end{equation*} 
			where $\xi_i$ are i.i.d.\ random variables with common distribution $p$. Let us indicate with $m=\sum_{k=0}^\infty k p(k)$
			the first moment of the distribution of the offspring. It classifies GWPs into three classes: subcritical if $m<1$,
			supercritical if $m>1$ and critical if $m=1$. The following characteristic feature of GWPs is the branching property.
			Let us call $_{(j)}Z$ the GWP starting with $j$ individuals, i.e.\ $_{(j)}Z_0=j$ almost surely.
			Then the GWP is the only discrete-time and discrete-space Markov process such that for all $j,k\geq 0$,
			\begin{equation}
				\label{branch_property}
				_{(j+k)}Z\stackrel{d}{=} {_{(j)}{Z}^{(1)}}+ {_{(k)}{Z}^{(2)}},
			\end{equation}
			where $Z$, $Z^{(1)}$ and $Z^{(2)}$ are independent GWPs with the same offspring distribution. From a modelling point of view,
			this property underlines the fact that each individual in the population reproduces independently from the others according
			to the same offspring distribution $p$.
			
			Since the seminal works of Ji\v{r}ina and Lamperti \cite{Jir58, Lam67CSBP, Lam67}, there has been interest in defining branching processes
			in a continuous state-space setting and in identifying them as scaling limits of GWPs.
			The simplest way to extend the definition of branching processes to describe the evolution in continuous time of a population
			with values in $\mathbb{R}^+$ is by means of the branching property.
			Indeed, we define the CSBPs as the continuous time-continuous space processes satisfying
			an analogue of the branching property \eqref{branch_property} as follows.
			Rigorously, a stochastic process $X=\{X(t):t\geq 0\}$ is a CSBP if it is a Markov process characterized by a family of transition kernels
			$\{P_t(x,dy), \, t\geq0, \, x \in  \mathbb{R}^+\}$ satisfying, for all $t>0$ and $x,x'\in \mathbb{R}^+$ (see e.g.\ \cite{Lam67CSBP}),
			\begin{equation*}
				P_t(x,\cdot)*P_t(x',\cdot)=P_t(x+x',\cdot).
			\end{equation*}
			Let $\mathbb{D}(\mathbb{R}^+)$ be the set of c\`adl\`ag functions  defined on $\mathbb{R}^+$ with
			values on $\mathbb{R}^+$, a CSBP is a random variable in  $\mathbb{D}(\mathbb{R}^+)$. From now on we will consider
			$\mathbb{D}(\mathbb{R}^+)$ as a topological space endowed with the usual Skorokhod topology. For a complete description see \cite{JaSh13}.
			Further, we denote by $\mathbf{E}_x$ the expectation with respect to the law of the process $X$ starting from the initial value
			$x\in\mathbb{R}^+$. Let us underline that CSBPs are characterized by their Laplace transform, i.e. for all
			$\lambda>0$ we have
			\begin{equation*}
				\mathbf{E}_x\left[e^{-\lambda X(t)}\right]=
				\int_{0}^{\infty}e^{-\lambda y}P_t(x,dy)=e^{-x \nu_t(\lambda)},
			\end{equation*}
			where $\nu_t(\lambda)$ is the unique nonnegative solution to the equation
			\begin{equation}
				\label{exponent}
				\nu_t(\lambda)+\int_0^t\psi(\nu_s(\lambda))ds=\lambda.
			\end{equation}
			Here $\psi$ can be written as
			\begin{equation*}
				\psi(u)=bu+cu^2+\int (e^{-zu}-1+zu)m(dz),
			\end{equation*}
			where $b\in\mathbb{R}$, $c\geq0$ and $m$ is a $\sigma$-finite measure on $(0,\infty)$ such that $\int(z\wedge z^2)m(dz)<\infty$.
			The function $\psi$ is called the branching mechanism of the CSBP and, at the same time, it is the characteristic function of a L\'evy
			process without negative jumps killed at the first time it becomes negative.
			This identifies a relationship between CSBPs and the latter class of L\'evy processes that is known as
			Lamperti transform. Indeed, also the converse property holds true, i.e.\ the characteristic function $\psi$ of every L\'evy process
			without negative jumps and killed at zero is the branching mechanism of a CSBP (see  \cite{LeG99, Sil68}).
			The branching mechanism $\psi$, in addition to the Lamperti transform, plays a role in classifying CSBPs in three categories: critical,
			subcritical and supercritical processes. A CSBP is supercritical when $b<0$, critical when $b=0$ and subcritical when $b>0$. Moreover, in
			\cite{Lam67} we see that the parameters of $\psi$ appear in the explicit form of the first two moments of a CSBP $X$, that is
			\begin{align}
				\label{momenti}
				\mathbf{E}_x[X(t)] & = x e^{-bt}, \\
				\mathbf{E}_x[X(t)^2] & =
				\begin{cases}
					x^2 + x \tilde{\beta} t, & b=0, \\
					x^2e^{-2bt}-\frac{\tilde{\beta}x}{b}\left( e^{-2bt}-e^{-bt}\right), & b\neq0, \notag
				\end{cases}
			\end{align}
			where $\tilde{\beta}=\left(2c + \int_0^{\infty} u^2 m(du)\right )$.
			Let us mention that, despite CSBPs in general have
			discontinuous sample paths, the Feller branching diffusion (introduced in  \cite{Fel51}) which is a CSBP whose branching
			mechanism has the form $\psi(u)=bu+cu^2$, exhibits continuous sample paths.

			Results on convergence of suitably rescaled sequences of GWPs to CSBPs appeared first in \cite{Lam67} and, subsequently,
			in several other papers such as \cite{AlSh83, Gri74, Li10}. In the following we briefly state the results and the approach.
			Consider a sequence of GWPs 
			\begin{equation*}
				Z^{(k)}=\{Z^{(k)}_{n}\}_{n \in\mathbb{N}}, \qquad k= 1,2,3,\dots,
			\end{equation*}
			defined through their offspring distribution $p^{(k)}$. Define a sequence of positive integers $\{c_k\}_{k\in\mathbb{N}}$,
			tending to infinity, and  the Markov process 
			\begin{equation}
				\label{succlimite}
				\{X_k(t)\}_{t\geq 0}=\left \{ \frac{Z^{(k)}_{\lfloor kt\rfloor}}{c_k} \right \}_{t\geq0}, \qquad
				Z^{(k)}_0=c_k\, \quad a.s.,
			\end{equation}
			where for each $y$~$\in$~$\mathbb{R}$ we denote with $\lfloor y \rfloor$ its integer part.
			If the sequence of processes $\{X_k\}_{k\geq0}$ has a weak limit in the sense of finite-dimensional distributions,
			then this limit is a CSBP. This result is extended to convergence in the Skorokhod space $\mathbb{D}(\mathbb{R}^+)$ in \cite{Gri74}.
			Briefly, let $\mu_k$ be the probability measure on $\left\{-1/c_k,0,1/c_k,2/c_k,\dots\right\}$
			defined as follows: for all $n\in\mathbb{N}$,
			\[
				\mu_k\left(\frac{n-1}{c_k}\right)=p^{(k)}(n),
			\]
			and assume that there exists
			a measure $\mu$ such that $(\mu_k)^{*k c_k}\rightarrow \mu$, weakly as $k\rightarrow\infty$. Then the sequence of GWPs $Z^{(k)}$
			with offspring distribution $p^{(k)}$ and normalized as in \eqref{succlimite}, has a weak limit as a sequence of random variables on
			$\mathbb{D}(\mathbb{R}^+)$; this limit, say $X$, is a CSBP with initial condition $X(0)=1$ almost surely.
			Conversely, for every CSBP $X$ there exists a sequence of GWPs $\{Z^{(k)}\}_{k\in\mathbb{N}}$ and a sequence of positive
			integers $\{c_k\}_{k\in\mathbb{N}}$ such that $X$ is the limit of the sequence rescaled as in \eqref{succlimite}.

		\subsection{Random times and stable subordinators}\label{waiting_times}

			Let us consider a sequence i.i.d.\ real positive random variables
			$J_1,J_2,\dots$ representing for us a sequence of random waiting times.
			We define for all $n\geq 0$ the process $T_n\colon = \sum_{i=1}^nJ_i$.
			Its inverse, for all $t\geq 0$, is the renewal process
			\begin{equation}
				\label{N_t}
				N_t\colon =\max \{n\geq0 \colon  T_n\leq t\}.
			\end{equation}
			We assume now that these waiting times belong to the strict domain of attraction of a certain
			completely skewed stable random variable $D$
			with stability parameter $\beta\in(0,1)$. Note that due to the extended central limit theorem there exists a sequence
			$\{b_n\}_{n\geq0}$ such that the following convergence holds in distribution \cite{MeSi12}:
			\begin{equation*}
				b_n T_n \Rightarrow D.
			\end{equation*}
			As a consequence, the rescaled process $\left\{b_nT_{\lfloor nt \rfloor}\right\}_{t\geq0}$ converges in
			$\mathbb{D}(\mathbb{R}^+)$ to the stable subordinator $\{D(t)\}_{t\geq0}$ of parameter $\beta$, i.e.\ a L\'evy process such
			that $D(t)\stackrel{d}{=}t^{1/\beta}D$ for all $t\geq0$ and with Laplace transform 
			\begin{align*}
				\mathbf{E}[e^{-sD(t)}]=e^{-s^{\beta}t}, \qquad s>0.
			\end{align*} 
			Similarly, the scaling limit for the renewal process $\{N_t\}_{t\geq0}$ is the hitting time process of $\{D(t)\}_{t\geq0}$,
			that we define below. Indeed, let $\{\tilde{b}_n\}_{n\geq0}$ be a regularly varying sequence with index $\beta$ such that
			$\lim_{n\rightarrow\infty}n b_{\lfloor\tilde{b}_n\rfloor}=1$, then the following limit holds:
			\begin{equation}
				\label{N->E}
				\left\{\frac{N_{nt}}{\tilde b_n}\right\}_{t\geq0}\Rightarrow  \{\mathcal{E}(t)\}_{t\geq0},
			\end{equation}  
			where the process  $\{\mathcal{E}(t)\}_{t\geq0}$ is known as the inverse $\beta$-stable subordinator, defined as
			\begin{equation*}
				\mathcal{E}(t):=\inf\{u>0 \colon D(u)>t\}, \qquad t \ge 0.
			\end{equation*}
			The process $\{\mathcal{E}(t)\}_{t\geq 0}$ is a non-Markov process with non-decreasing continuous sample paths
			and plays a role in models of phenomena exhibiting long memory; for instance $\mathcal{E}$ has a fundamental
			importance in the study of time-fractional sub-diffusions
			\cite{MeSt13}. Let us now denote by $h(u,t)$ the probability density function of the random variable $\mathcal{E}(t)$ for a fixed
			time $t\geq0$. It is known that the Laplace transform of $h(u,t)$ w.r.t.\ variable $t$ is
			\begin{equation}
				\label{laplace_h}
				\mathcal{L}(h(u,t))(s)=s^{\beta-1}e^{-us^{\beta}}, \qquad s>0.
			\end{equation}
			Furthermore, the Laplace transform w.r.t.\ variable $u$ reads
			\begin{equation*}
				\mathbf{E}[e^{-\lambda \mathcal{E}(t)}] = E_\beta(-t\lambda^\beta), \qquad \lambda > 0,
			\end{equation*}
			where $E_\nu(z)$ is the Mittag--Leffler function defined as the convergent series
			\begin{align}\label{def_mittag_leffl}
				E_\nu(x) = \sum_{r=0}^\infty \frac{x^r}{\Gamma(r \nu+1)}, \qquad x \in \mathbb{R}, \: \nu>0.
			\end{align}
			Moreover, the dynamic of this process is driven by a fractional evolution, i.e.\ $h(u,t)$ evolves according a governing equation
			involving a fractional derivative in the $t$ variable and a first order derivative in the $u$ variable. This means that $h(u,t)$,
			for all $t\geq0$ and $u>0$, solves the fractional PDE
			\begin{equation*} 
				\partial_t^{\beta}h(u,t)=-\partial_u h(u,t),
			\end{equation*}
			where $\partial_t^{\beta}$ stands for the D\v{z}rba\v{s}jan--Caputo fractional derivative of order $\beta$ which is defined as
			follows:
			\begin{defn}[D\v{z}rba\v{s}jan--Caputo derivative]
				\label{capu}
			    Let $\alpha>0$, $m = \lceil \alpha \rceil$, and $f \in AC^m(0,b)$.
			    The D\v{z}rba\v{s}jan--Caputo derivative of order $\alpha>0$ is defined as
			    \begin{equation}
			        \label{Capu}
			        \partial_t^{\alpha} f(t)= \frac{1%
			        }{\Gamma(m-\alpha)}\int_a^{t}(t-s)^{m-1-\alpha}\frac{\textup{d}^m}{\textup{d}s^m}f(s) \, \textup{d}s.
			    \end{equation}
			\end{defn}

	\section{Time fractional branching processes}\label{Limite}

		Following the approach used in \cite{MeSc04} to define time-fractional diffusions, we introduce in this section a time-changed
		GWP and we prove that there exists a certain scaling such that its limit is exactly a time-changed CSBP.

		\subsection{Time-changed GWPs}\label{GW-modif}

			Let us consider a GWP $Z$, we want to define a GWP with random waiting times between successive generations.
			Further, let $\{J_1,J_2,\dots\}$ be a sequence of i.i.d.\ random variables. The time-changed GWP is defined as 
			\begin{equation}
				\label{time_changedGW}
				\mathcal{Z}_t\colon = Z_{N_t},
			\end{equation}
			for all $t\geq0$, where $N_t$, independent of $Z$, is the renewal process defined in \eqref{N_t}.
			As soon as the waiting times $\{J_1,J_2,\dots\}$ are not exponentially distributed, the process $\{\mathcal{Z}_t\}_{t\geq0}$
			is not a Markov process anymore. The following property holds.

			\begin{prop}[Branching inequality]
				\label{lennovvo}
				We have, for all $j,$ $k\in\mathbb{N}$ and all $\lambda\geq 0$,
				\begin{equation}
					\label{Dis_branch_discrete}
					\mathbf{E}_{j+k}\left[ e^{-\lambda \mathcal{Z}_t}\right]\geq \mathbf{E}_{j}\left[
					e^{-\lambda \mathcal{Z}_t}\right]\mathbf{E}_{k}\left[ e^{-\lambda \mathcal{Z}_t}\right].
				\end{equation}
			\end{prop}
			\begin{proof}
				Let us consider the function
				\begin{align}
					K_{j,k}(t)=\mathbf{E}_{j+k}\left[ e^{-\lambda \mathcal{Z}_t}\right]- \mathbf{E}_{j}\left[ e^{-\lambda \mathcal{Z}_t}\right]
					\mathbf{E}_{k}\left[ e^{-\lambda \mathcal{Z}_t}\right].
				\end{align}
				By taking conditional expectation with respect to $N(t)$, we get
				\begin{align}
					\label{eq2}
					K_{j,k}(t)= {} & \mathbf{E}\left[\mathbf{E}_{j+k}\left[\left. e^{-\lambda Z_{N(t)}}\right | N(t)\right] \right] \\
					& -\mathbf{E}\left[\mathbf{E}_{j}\left( \left.e^{-\lambda Z_{N(t)}}
					\right | N(t) \right)\right]\mathbf{E}\left[\mathbf{E}_{k}\left(
					\left.e^{-\lambda Z_{N(t)}} \right | N(t) \right) \right]. \notag
				\end{align}
				Observe that $\mathbf{E}_{x}\left[ \left.e^{-\lambda Z_{N(t)}} \right | N(t) \right]$ and $\mathbf{E}_{y}
				\left[ \left.e^{-\lambda Z_{N(t)}} \right | N(t) \right]$ are positively correlated being functions of the same
				random variable $N(t)$. Indeed, if we denote by $f$ the generating function of
				the GWP $Z$ and $f_n$ its $n$-th iterate, we know that we have
				\begin{align*}
				\mathbf{E}_{x}\left[ \left.e^{-\lambda Z_{N(t)}} \right | N(t) \right]&=f_{N(t)}\left(e^{-\lambda}\right)^{x};\\
				\mathbf{E}_{y}\left[ \left.e^{-\lambda Z_{N(t)}} \right | N(t) \right]&=f_{N(t)}\left(e^{-\lambda}\right)^{y}.
				\end{align*}
				Hence
				\begin{equation*}
					K_{j,k}(t) = \text{Cov}(f_{N(t)}\left(e^{-\lambda}\right)^{j}, f_{N(t)}\left(e^{-\lambda}\right)^{k}).
				\end{equation*}
				Being positive powers of the same function of $N(t)$, positive correlation follows
				and we obtain the inequality \eqref{Dis_branch_discrete}.
			\end{proof}

			\begin{remark}
				In a GWP, when we start with an initial population $Z_0=j+k$, the number of individuals $Z_n$ in the $n$-th generation
				is the sum of two independent copies of the process with initial size equal to $j$ and $k$ respectively.
				By introducing a random time change between the generations, we create a positive correlation between the sizes of subgroups
				of a given initial population.
			\end{remark}

			\begin{remark}
				In the special case of deterministic time-change, that is when $\mathbb{P}(N_t=l)=1$, with $g \colon \mathbb{N}
				\to \mathbb{N}$, $l=g([t])$ a suitable non decreasing function, the inequality \eqref{Dis_branch_discrete} becomes the classical
				equality that expresses the branching property of GWPs, i.e.
				\begin{equation}
					\label{eq1}
					\mathbf{E}_{j+k}\left[ e^{-\lambda Z_l}\right]=\mathbf{E}_{j}\left[
					e^{-\lambda Z_l}\right]\mathbf{E}_{k}\left[ e^{-\lambda Z_l}\right],
				\end{equation}
				for all $j,$ $k\in\mathbb{N}$ and all $\lambda\geq 0$.
			\end{remark}

		\subsection{Time-changed CSBP and scaling limit}\label{violin}

			Let us consider a CSBP $X$ and an inverse $\beta$-stable subordinator $\mathcal{E}$ independent of $X$. Consider the time-changed
			process 
			\begin{equation*}
				\mathcal{X}(t)\colon= X(\mathcal{E}(t)),
			\end{equation*}
			for all $t\geq0$.
			It is possible to show that there exists a sequence of time-changed GWPs $\{\mathcal{Z}^{(n)}_t\}_{t\geq0}$,
			such that, suitably rescaled, converges to the process $\{\mathcal{X}(t)\}_{t\geq0}$ in $\mathbb{D}([0,\infty))$. 

			\begin{theorem}
				\label{theor_conv} Let $\{X(t)\}_{t\geq0}$ be a CSBP and $\{\mathcal{E}(t)\}_{t\geq0}$ be the inverse of a $\beta$-stable
				subordinator, $\beta\in(0,1]$, independent of $\{X(t)\}_{t \ge 0}$.
				Consider the process $\{\mathcal{X}(t):=X(\mathcal{E}(t))\}_{t\geq0}$; there exists a sequence of time-changed GWPs
				$\{\mathcal{Z}^{(n)}_t\}_{t\geq0}$ and two increasing sequences $\{\tilde{b}_n\}_{n\geq0}$ and  $\{c_n\}_{n\geq0}$
				with $\lim_{n\rightarrow\infty}\tilde{b}_n=\lim_{n\rightarrow\infty}c_n=\infty$, such that for $n\rightarrow\infty$
				\begin{equation}
					\label{limite_scaling}
					\left\{\frac{\mathcal{Z}^{(\tilde{b}_n)}_{nt}}{c_{\tilde{b}_n}}\right\}_{t\geq0} \Longrightarrow \{\mathcal{X}(t)\}_{t\geq0},
				\end{equation}
				where the convergence is in $\mathbb{D}([0,\infty))$. 
			\end{theorem}

			\begin{proof}
				Consider $J_1, J_2. \dots$, i.i.d.\ waiting times in the domain of attraction of a stable law of index $\beta$,
				(see Section~\ref{waiting_times}). Then there exists a sequence of positive real numbers $\{\tilde b_n\}$, diverging to
				infinity, such that the limit \eqref{N->E} holds. At the same time, we consider a sequence of GWPs $\{Z^{(k)}\}_{k\geq0}$ such
				that \eqref{succlimite} holds. Since the waiting times and the GWPs are independent, it  follows that,
				for all $n$ $\geq 0$,
				\begin{equation*}
					\left( X_{\tilde{b}_n}(t), \frac{T(nt)}{b_n} \right ) =\left( \frac{Z^{(\tilde{b}_n)}(\lfloor\tilde{b}_nt\rfloor)}{c_{\tilde{b}_n}},
					\frac{T(nt)}{b_n} \right) \Longrightarrow \left( X(t), D(t) \right)
				\end{equation*}
				in the product space $\mathbb{D}([0,\infty))\times \mathbb{D}([0,\infty))$,
				where $Z^{(k)}(0)/c_k \rightarrow x$, $X(t)$ is a {CSBP} with transition semigroup $P_t(x,\cdot)$ and $D(t)$ is the stable
				subordinator of parameter $\beta$. Let us write $\mathbb{D}_{\uparrow, u}(\mathbb{R}^+)$ for the subset of unbounded non decreasing
				c\`adl\`ag functions and $\mathbb{D}_{\uparrow\uparrow, u}(\mathbb{R}^+)$ for the subset of unbounded strictly increasing ones.
				We see that, for all $n\geq0$, the pair
				\begin{align*}
					\left( \frac{Z^{(\tilde{b}_n)}(\lfloor\tilde{b}_nt\rfloor)}{c_{\tilde{b}_n}}, \frac{T(nt)}{b_n} \right)
				\end{align*}
				belongs to the product space
				$\mathbb{D}(\mathbb{R}^+)\times \mathbb{D}_{\uparrow, u}(\mathbb{R}^+)$ and the limit $\left( X(t) , D(t) \right)$ belongs
				to $\mathbb{D}(\mathbb{R}^+)\times \mathbb{D}_{\uparrow\uparrow, u}(\mathbb{R}^+)$. Then, following the approach in~\cite{HeSt11}, we
				define the function $\Psi\colon \mathbb{D}(\mathbb{R}^+)\times \mathbb{D}_{\uparrow, u}(\mathbb{R}^+) \rightarrow
				\mathbb{D}(\mathbb{R}^+)\times \mathbb{D}(\mathbb{R}^+)$ mapping $(x(t),d(t))$ to $(x(e(t)), t)$,
				where $e(t)$ is the inverse of $d(t)$. In general the function $\Psi$ is not continuous, however, in our case, since  the limit point
				$(X(t),D(t))$ actually belongs to $\mathbb{D}(\mathbb{R}^+)\times \mathbb{D}_{\uparrow\uparrow, u}(\mathbb{R}^+)$, as stated in
				\cite[Proposition~2.3]{HeSt11}, the function $\Psi$ is continuous at $(X(t),D(t))$. This implies that the following limit holds,
				where $\pi_1$ is the projection on the first coordinate,
				\begin{align*}
					\frac{\mathcal{Z}^{(\tilde{b}_n)}_{nt}}{c_{\tilde{b}_n}} & = X_{\tilde{b}_n}(N(nt))=\pi_1\left(\Psi\left( X_{\tilde{b}_n}(t),
					\frac{T(nt)}{b_n} \right)\right) \\
					& \Longrightarrow \pi_1\left(\Psi\left(X(t) , D(t) \right)\right).
				\end{align*}
				This proves \eqref{limite_scaling}.
			\end{proof}

	\section{Some properties of the time-fractional CSBP}\label{sec_properties}

		In the previous section we have characterized the process $\{\mathcal{X}(t)\}_{t\geq0}$ as the limit of a rescaled sequence of
		time-changed GWPs, where in the discrete case the time between two generations is substituted by random variables that
		produce a slowed-down dynamics. In the limit this is modeled by the inverse stable subordinator. We are now interested in capturing
		the main features of the time-changed process $\mathcal{X}$ and in underlining the differences between it and the classical CSBP.
		Note that the tree structure underlying Markov branching processes and CSBPs, although randomly stretched and squashed, it is
		still a characterizing feature of the corresponding time-changed processes.

		\subsection{Branching property}

			Let us consider $\beta\in(0,1]$. We expect the time-changed CSBP $\{\mathcal{X}(t)\}_{t\geq0}$ to satisfy the classical
			branching property only when $\beta=1$. Indeed, in general, it holds
			\begin{equation}
				\label{dis_branching_cont}
				\mathbf{E}_{x+y}[e^{-\lambda \mathcal{X}(t)}]\geq \mathbf{E}_{x}[e^{-\lambda \mathcal{X}(t)}]
				\mathbf{E}_{y}[e^{-\lambda \mathcal{X}(t)}]
			\end{equation} 
			and 
			\begin{equation}
				\label{lim_beta}
				\lim_{\beta\rightarrow 1}\mathbf{E}_{x+y}[e^{-\lambda \mathcal{X}(t)}]=\mathbf{E}_{x}[e^{-\lambda \mathcal{X}(t)}]
				\mathbf{E}_{y}[e^{-\lambda \mathcal{X}(t)}].
			\end{equation}
			Similarly to Section~\ref{GW-modif}, this is based on the following:
			\begin{align*}
				& \mathbf{E}_{x+y}[e^{-\lambda \mathcal{X}(t)}]- \mathbf{E}_{x}[e^{-\lambda \mathcal{X}(t)}]
				\mathbf{E}_{y}[e^{-\lambda \mathcal{X}(t)}] \\
				& \text{\small $= \mathbf{E}\left[\mathbf{E}_{x+y}\left(e^{-\lambda X(\mathcal{E}(t))}|\mathcal{E}(t)\right)\right]-
				\mathbf{E}\left[\mathbf{E}_{x}
				\left(e^{-\lambda X(\mathcal{E}(t))}|\mathcal{E}(t)\right)\right]\mathbf{E}\left[ \mathbf{E}_{y}
				\left(e^{-\lambda X(\mathcal{E}(t))}|\mathcal{E}(t)\right)\right]$} \\
				& = \mathbf{E}\left[e^{-(x+y)\nu_{\mathcal{E}(t)}(\lambda)}\right]
				- \mathbf{E}[e^{-x\nu_{\mathcal{E}(t)}}] \mathbf{E}[e^{-y\nu_{\mathcal{E}(t)}}] \\
				& = \text{Cov}(e^{-x\nu_{\mathcal{E}(t)}(\lambda)},e^{-y\nu_{\mathcal{E}(t)}(\lambda)}).
			\end{align*}
			Then, by positive correlation, we see that $\text{Cov}(e^{-x\nu_{\mathcal{E}(t)}(\lambda)},e^{-y\nu_{\mathcal{E}(t)}(\lambda)})\geq0$.
			Moreover, since $\mathcal{E}(t)\rightarrow t$ in distribution as $\beta\rightarrow1$, by dominated convergence and the continuity
			of $\nu_{t}(\lambda)$ in $t$ (which is a consequence of \eqref{exponent}), we see that 
			\begin{align*}
				\lim_{\beta\rightarrow1}\text{Cov}(e^{-x\nu_{\mathcal{E}(t)}(\lambda)},e^{-y\nu_{\mathcal{E}(t)}(\lambda)})=0,
			\end{align*}
			proving \eqref{lim_beta}. Note that the random time-change introduces a positive correlation between the evolution of  the
			subgroups of the initial population that is not present in the classical CSBP. However, for any $\beta \in (0,1)$,
			we still have a conditional branching property, i.e.
			\begin{align*}
				\mathbf{E}[\mathbf{E}_{x+y}[e^{-\lambda X(\mathcal{E}(t))}|\mathcal{E}(t)]]= \mathbf{E}[\mathbf{E}_{x}[e^{-\lambda
				X(\mathcal{E}(t))}|\mathcal{E}(t)]\ \mathbf{E}_{y}[e^{-\lambda X(\mathcal{E}(t))}|\mathcal{E}(t)]].
			\end{align*}

		\subsection{First and second moment }\label{momenti_sec}
			
			Here we obtain the expression for the first and the second moment of the process $\{\mathcal{X}(t)\}$, when they exist.
			To this aim, we exploit the computations in~\cite{Lam67} for the explicit formula of first and second moment of a CSBP,
			see equation \eqref{momenti}, and the properties of the
			Mittag--Leffler function defined in \eqref{def_mittag_leffl}, see~\cite{HaErMaOb55}. 

			\begin{theorem}
				\label{mom1csbpcomp} Let $\{X(t)\}_{t\geq0}$ be a CSBP with Laplace exponent $v_t(\lambda)$ and branching mechanism $\psi(z)$ and
				let $\{\mathcal{E}(t)\}_{t\geq0}$ be an inverse stable subordinator with index $\beta\in(0,1)$ and with density function
				$h(\cdot,t)$, for every fixed time $t\geq0$. If $\frac{\partial v_t(0^+)}{\partial \lambda}$ exists finite and
				$\psi'(0^+)=b\geq \sigma_h$, where $\sigma_h$ is the abscissa of convergence for the Laplace transform of the function $h(\cdot,t)$,
				then the time-changed process $\{\mathcal{X}(t)\}_{t\geq0}$ has finite first moment that takes the form
				\begin{equation}
					\label{formulamediaCSBPcomp}
					\mathbf{E}_x[\mathcal{X}(t)]=x E_{\beta}(-bt^{\beta}), \qquad t\geq0.
				\end{equation}
			\end{theorem}
			
			\begin{proof}
				For the independence of $\{\mathcal{E}(t)\}_{t\geq0}$ from $\{X(t)\}_{t\geq0}$, together with the formula for the first moment of a
				CSBP, we get
				\begin{equation}
					\label{med}
					\mathbf{E}_x[\mathcal{X}(t)]=\int_0^{\infty}\mathbf{E}_x[X(u)]h(u,t)du=\int_0^{\infty}xe^{-bu}h(u,t)du.
				\end{equation}
				Since $b\geq\sigma_h$ the last integral is finite and it is essentially the Laplace transform of $h(u,t)$
				with respect to the variable $u$. To obtain an explicit form of the integral, we apply again the Laplace transform
				to \eqref{med}, this time with respect to the variable $t$, obtaining
				\begin{equation*}
					\mathcal{L}\left[ \mathbf{E}_x[\mathcal{X}(\cdot)]\right](\mu)= x \int_0^{\infty}  e^{-bu}  \int_0^{\infty} e^{-\mu t}h(u,t)  dt du.
				\end{equation*}
				Formula \eqref{laplace_h} leads to
				\begin{align*}
					\mathcal{L}\left[ \mathbf{E}_x[\mathcal{X}(\cdot)]\right](\mu)
					 & = x \mu^{\beta -1}  \int_0^{\infty}  e^{-u(b+\mu^{\beta})} du \\
					 & = x \frac{\mu^{\beta-1}}{\mu^{\beta} +b}.
				\end{align*}
				Since the latter expression is the Laplace transform of the Mittag--Leffler function $E_{\beta}(-bt^\beta)$ we immediately
				obtain formula \eqref{formulamediaCSBPcomp}.
			\end{proof}

			\begin{figure}[ht!] 
				\minipage{0.45\textwidth}
			    	\includegraphics[width=\linewidth]{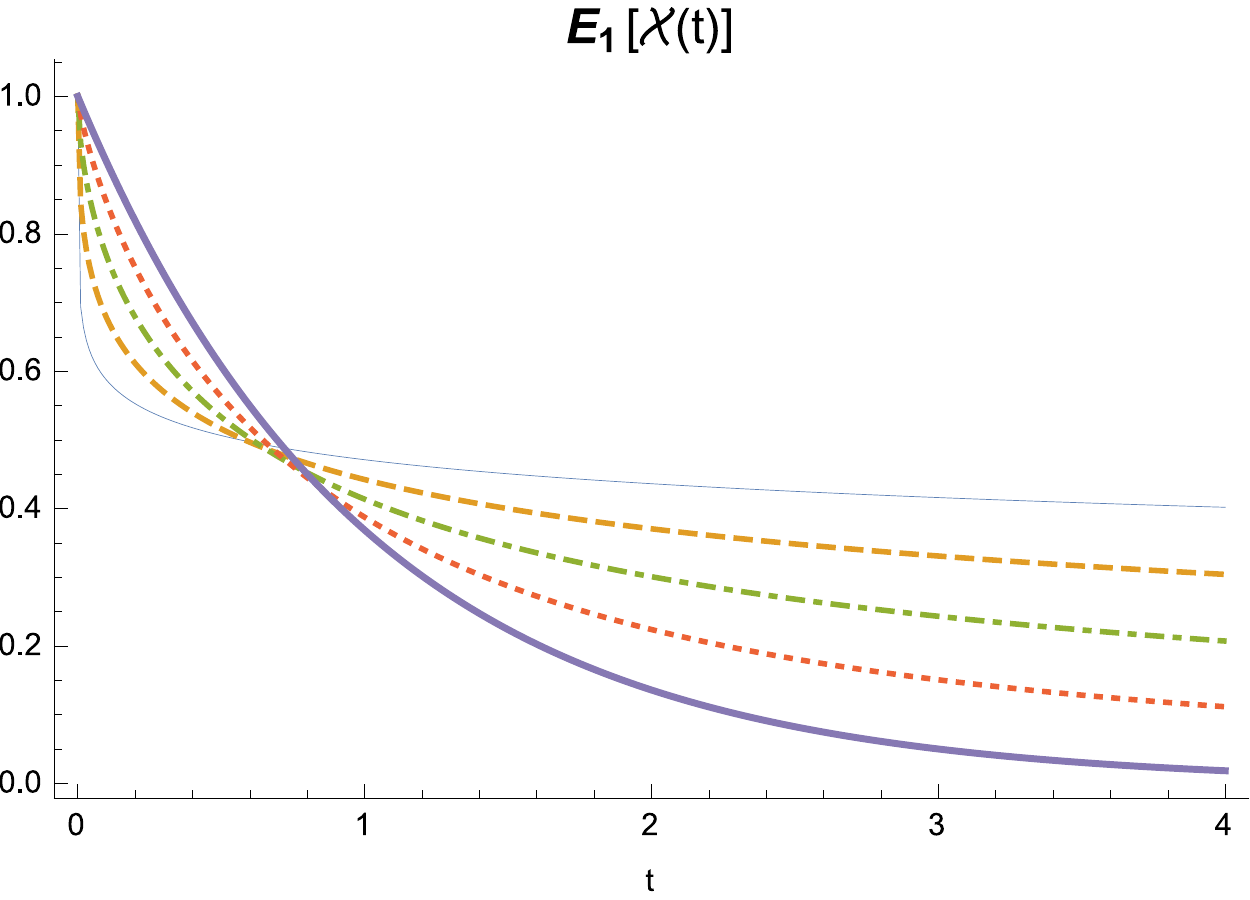}
				\endminipage \hfill
				\minipage{0.45\textwidth}
			    	\includegraphics[width=\linewidth]{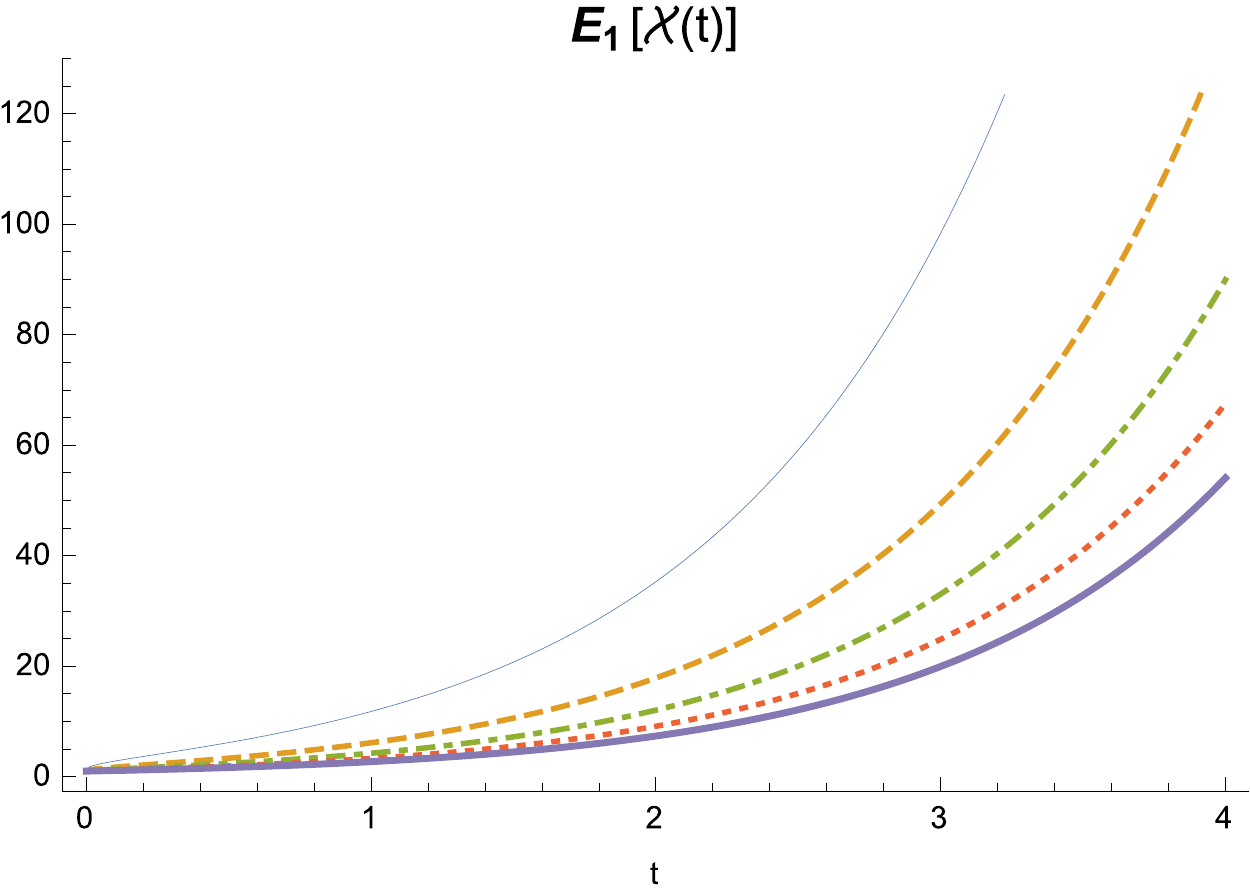}
				\endminipage \hfill
				\minipage{0.1\textwidth}
					\includegraphics[width=\linewidth]{legenda.pdf}
				\endminipage
				\caption{\label{fig:media}Plots of the first moment $\mathbf{E}_1[\mathcal{X}(t)]$ of a time-changed CSBP for $t\in[0,4]$
					and different values of $\beta$, from $0.2$ to $1$. The time-changed CSBP has initial condition $\mathcal{X}(0)=1$ a.s.;
					on the left the subcritical case with $b=1$, and on the right  the supercritical case with $b=-1$.}
			\end{figure}

			\begin{theorem}
				\label{mom2csbpcomp}
				Let $\{X(t)\}_{t\geq0}$ be a CSBP with Laplace exponent $v_t(\lambda)$ and branching mechanism $\psi(z)$ and
				let $\{\mathcal{E}(t)\}_{t\geq0}$ be an inverse stable subordinator with index $\beta\in(0,1)$ and with density function
				$h(\cdot,t)$ for every fixed time $t\geq0$. If $\frac{\partial v_t(0^+)}{\partial \lambda}$ and
				$\frac{\partial^2 v_t(0^+)}{\partial \lambda^2}$ exist finite and $\psi'(0^+)=b\geq\sigma_h$
				as in Theorem~\ref{mom1csbpcomp}, then the time-changed process $\{\mathcal{X}(t)\}_{t\geq0}$ has the following finite
				second moment:
				\begin{equation}
					\label{momsecondoCSBPcomp}
					\mathbf{E}_x\left[ \mathcal{X}(t)^2\right] =
					\begin{cases}
						x^2 + x \tilde{\beta} \frac{\Gamma(2)}{\Gamma(\beta+1)}t^{\beta}, & b=0, \\
						x^2E_{\beta}(-2bt^{\beta})-\frac{\tilde{\beta}x}{b}\left( E_{\beta}(-2bt^{\beta})-E_{\beta}(-bt^{\beta})\right),
						& b\neq0,
					\end{cases}
				\end{equation}
				where $\tilde{\beta}=\left(2c + \int_0^{\infty} u^2 m(du)\right )$.
			\end{theorem}

			\begin{proof}
				Fix $t\geq0$, we divide the proof into two different cases.
				\begin{itemize}
					\item \emph{Case $b=0$:}
						We know that
						\begin{align*}
							\mathbf{E}_x\left[\mathcal{X}(t)^2\right] & =\mathbf{E}\left[ \mathbf{E}_x\left[\left.
							X(\mathcal{E}(t))^2\right | \mathcal{E}(t)\right] \right] \\
							& = \int_0^{\infty} (x^2+ x\tilde{\beta}u) h(u, t) du
							= x^2 + x \tilde{\beta} \mathbf{E}[\mathcal{E}(t)].
						\end{align*}
						It is known, see \cite{MeSc04}, Corollary 3.1, that the first moment of the process
						$\{\mathcal{E}(t)\}_{t\geq0}$, for a fixed time $t\geq0$, takes the form
						\begin{align*}
							\mathbf{E}\left[\left(\mathcal{E}(t)\right)\right]=\frac{\Gamma(2)t^{\beta}}{\Gamma(\beta+1)}.
						\end{align*}
						Hence we obtain
						\begin{align*}
							\mathbf{E}_x\left[\mathcal{X}(t)^2\right]=x^2 + x \tilde{\beta} \frac{\Gamma(2)t^{\beta}}{\Gamma(\beta+1)}.
						\end{align*}
					
					\item \emph{Case $b\neq0$:}
						In this case we write
						\begin{align*}
							\mathbf{E}_x\left[\mathcal{X}(t)^2\right] & = \int_0^{\infty} \left(x^2e^{-2bu}
							-\frac{\tilde{\beta}x}{b}(e^{-2bu}-e^{-bu}) \right) h(u, t) du\\
							& = x^2E_{\beta}(-2bt^{\beta})-\frac{\tilde{\beta}x}{b}\left( E_{\beta}(-2bt^{\beta})-E_{\beta}(-bt^{\beta})\right).
						\end{align*}
				\end{itemize}
			\end{proof}
 
			\begin{figure}[ht!]
				\centering
				\begin{minipage}[b]{\linewidth}
					\includegraphics[width=0.44\linewidth]{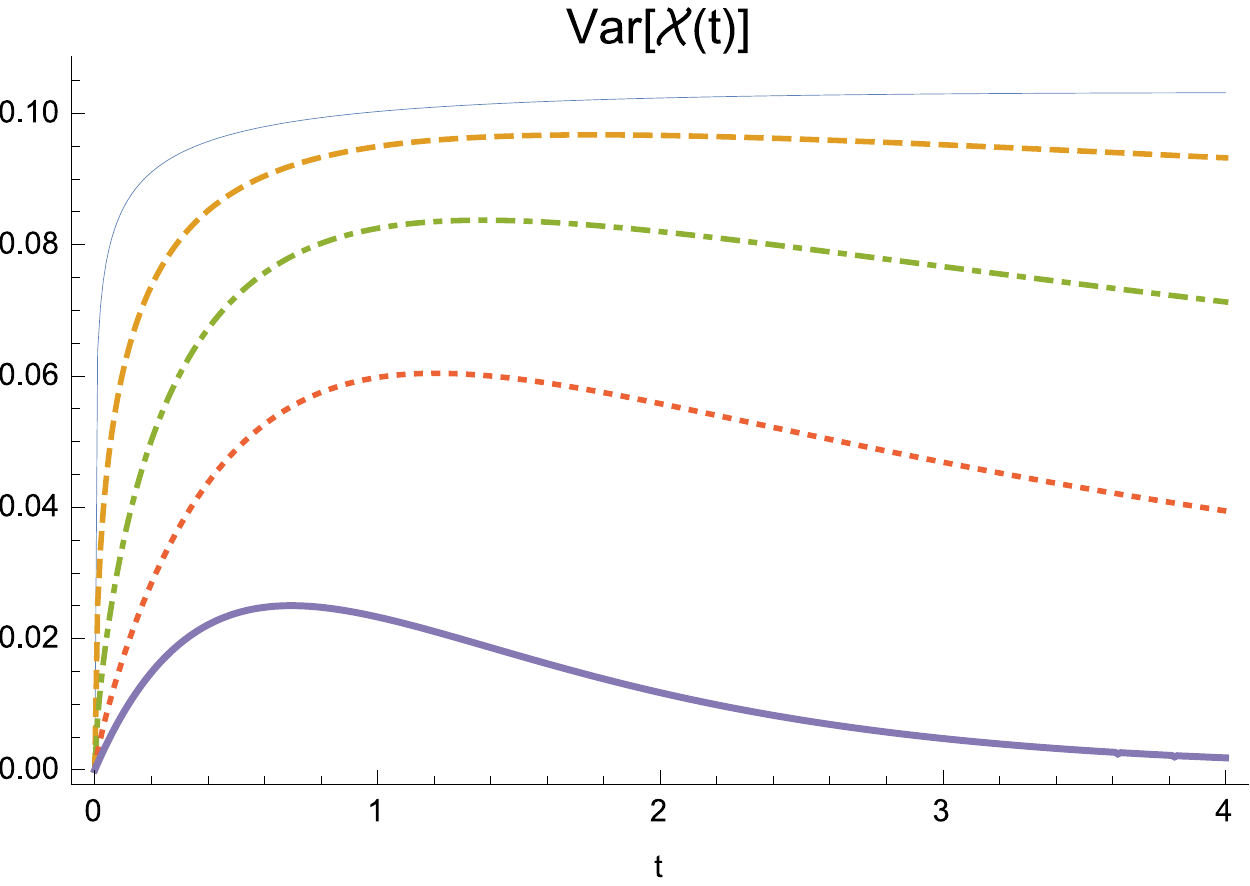}
				    \includegraphics[width=0.44\linewidth]{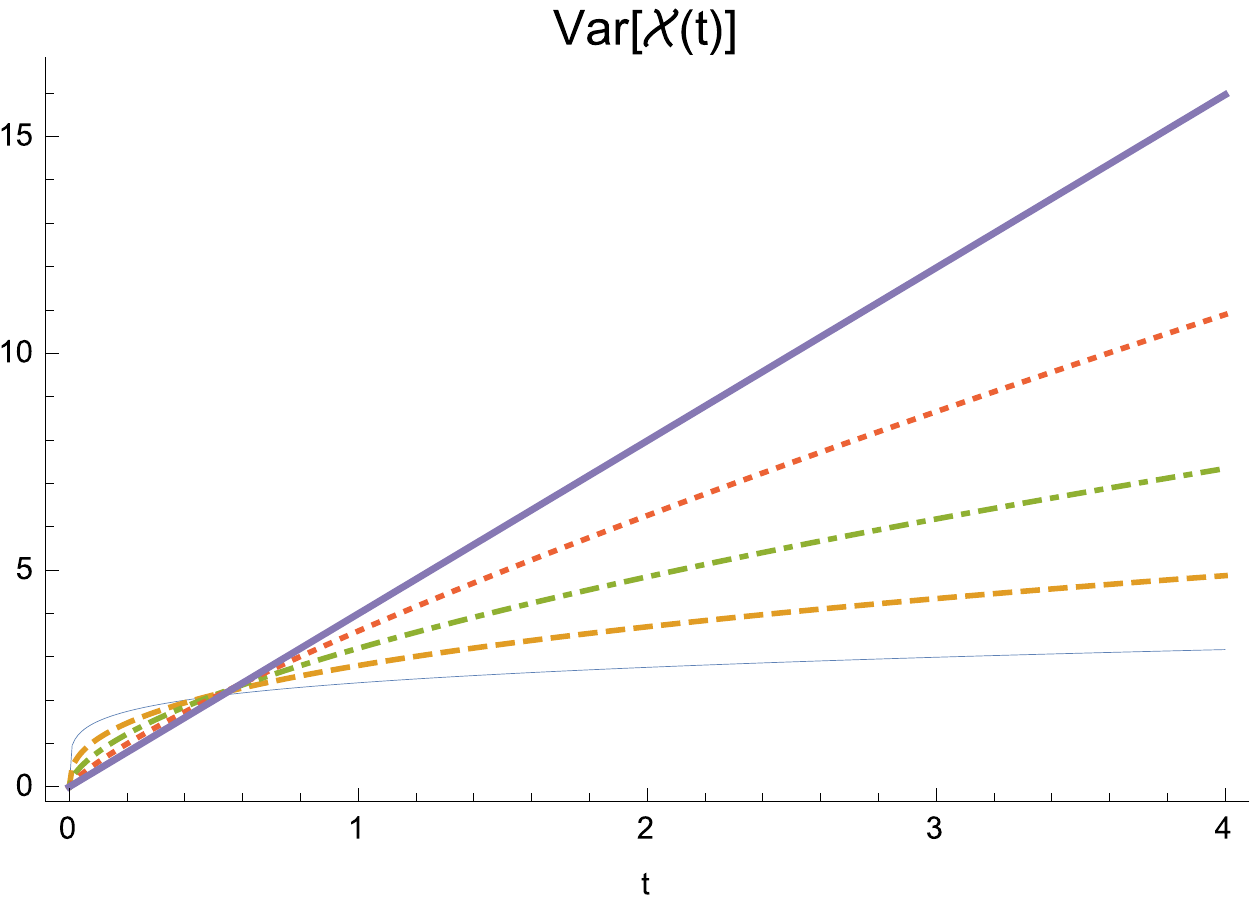}
				    \includegraphics[width=0.1\linewidth]{legenda.pdf}
				\end{minipage} \hfill
				\ \hspace{2mm} \hspace{3mm} \
				\begin{minipage}[b]{\linewidth}
					\includegraphics[width=0.44\linewidth]{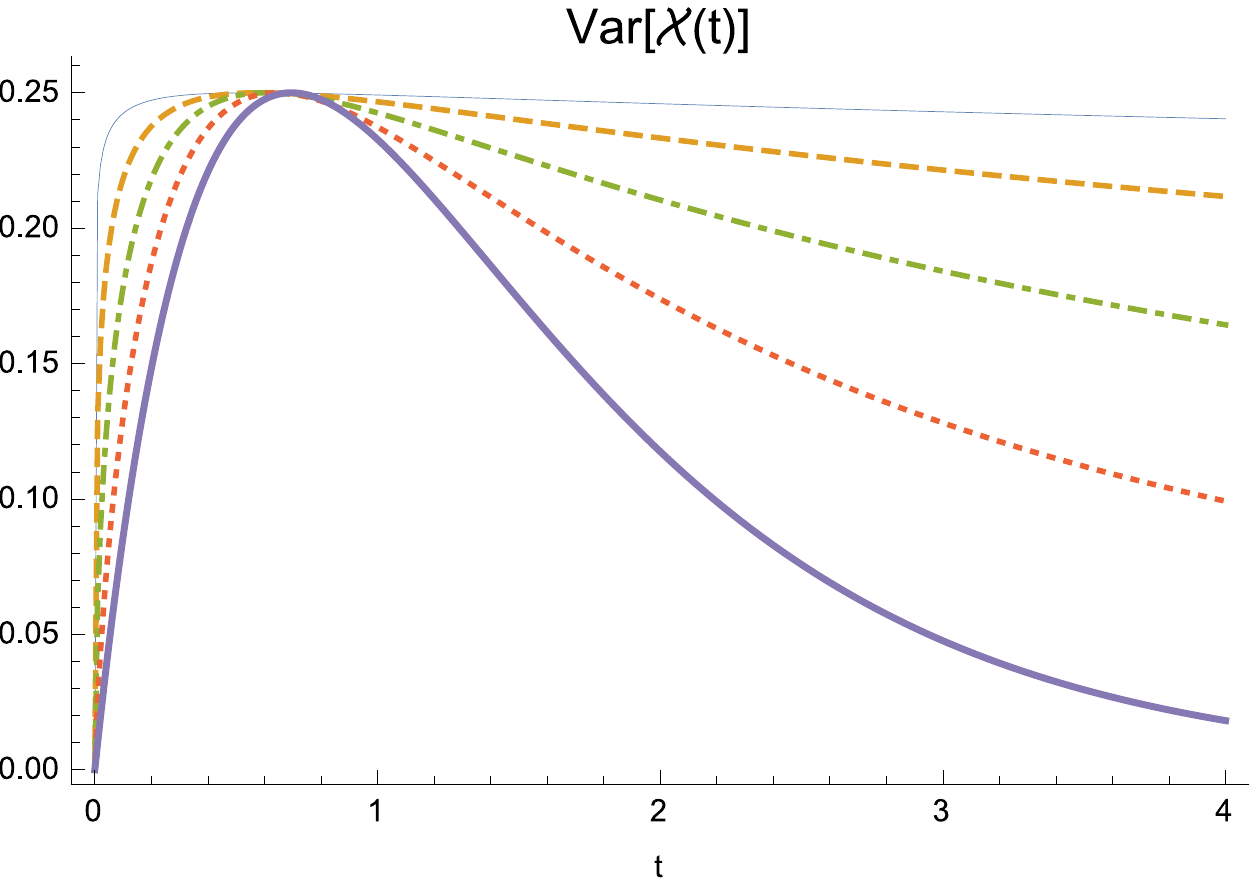}
				    \includegraphics[width=0.44\linewidth]{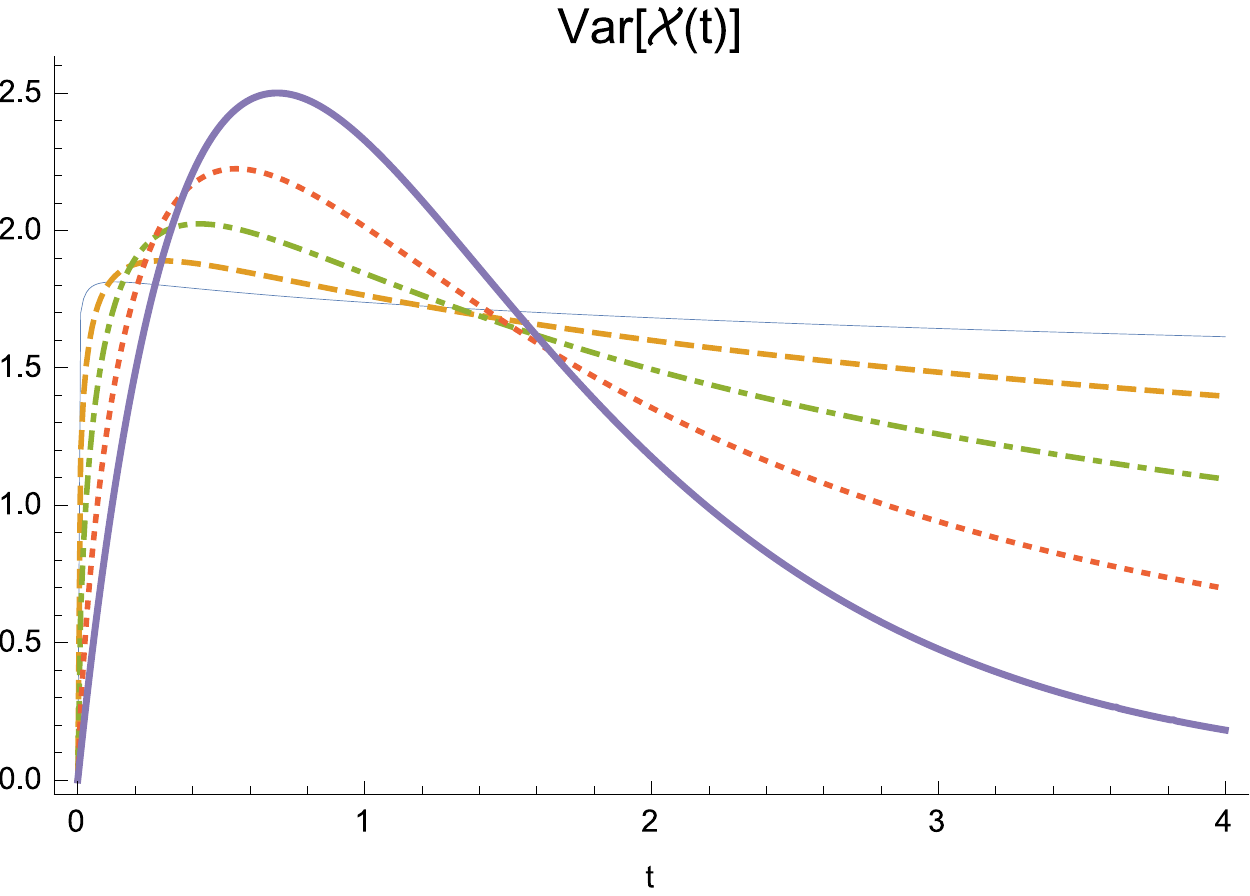}
				\end{minipage}
				\caption{\label{fig:var}Plots of $Var(\mathcal{X}(t))$ for $t\in [0,4]$ and
					$\beta$ from $0.2$  to $1$. The time-changed CSBP has initial condition
					$\mathcal{X}(0)=1$ a.s.\ and the pair of parameters $(b,\, \tilde \beta)$, clockwise from the upper-left, equal to $(1,\,0.1)$,
					$(1,\,0.5)$, $(1, \,10)$ and $(0,\,1)$, respectively.}
			\end{figure}

			Note that the Mittag--Leffler function is a generalization of the exponential function, with which it coincides for $\beta=1$.
			Comparing the moments of our generalized model, in~\eqref{formulamediaCSBPcomp} and~\eqref{momsecondoCSBPcomp}, to those of
			the CSBP in \eqref{momenti}, it is easy to see that the Mittag--Leffler function in the generalized case plays the same role
			as the exponential in the classical case. See in Figure~\ref{fig:media} and Figure~\ref{fig:var} the effect of the time-change
			on the mean and variance of the process $\mathcal{X}$.

		\subsection{Some examples}

			In the previous sections we have described in full generality the time-changed CSBP $\{\mathcal{X}(t)\}_{t\geq0}$,
			now let us focus on some specific cases of interest in order to better illustrate our framework.

			\subsubsection{Time-changed Feller branching diffusion}

				Consider the Feller branching diffusion \cite{Fel51} and recall that it is the diffusion process solving the SDE
				\begin{equation}
					\label{sdefeller}
					dX_t=-b X_t dt + \sqrt{2c X_t} dW_t
				\end{equation}
				where $W_t$ is a standard Brownian motion, $b$ $\in \mathbb{R}$ and $c>0$.
				This is the only diffusion process in the class of CSBPs and its corresponding Fokker--Plack equation is 
				\begin{equation*}
					\frac{\partial}{\partial t} p(y,t)=\frac{\partial }{\partial y} \left(byp(y,t)\right)
					+\frac{\partial^2}{\partial y^2}\left( cyp(y,t)\right).
				\end{equation*} 
				
				The scaling limit of GWPs that leads to Feller branching diffusion is well-known, see Pardoux \cite{Par08}
				for a nice review on it;
				this is one of the few cases in which this scaling scheme is known explicitly.

				Consider thus a time-changed Feller branching diffusion $\{\mathcal{X}(t)\}_{t\geq0}$ with stability parameter $\beta \in (0,1)$.
				Since the process $\{X(t)\}_{t\geq0}$ is a diffusion, its composition with $\{\mathcal{E}(t)\}_{t\geq0}$ fits in the
				framework of SDE driven by time-changed L\'evy processes, see \cite{HaKoUm12}. Therefore, it is possible to write an analogue
				of the Fokker--Planck equation solved by the marginal probability density function $m_x(y,t)$ of $\{\mathcal{X}(t)\}_{t\geq0}$.
				The following proposition shows that the equation involves D\v{z}rba\v{s}jan--Caputo derivatives of order $\beta\in(0,1)$,
				hence it classifies the time-changed Feller branching diffusion in the class of subdiffusions.

				\begin{prop}
					Let $\{\mathcal{X}(t)\}_{t\geq0}$ be a time-changed Feller branching diffusion with branching mechanism
					$\psi(u)=bu+cu^2$, for $b\in\mathbb{R}$ and $c>0$ and parameter $\beta \in (0,1)$. Let $\mathcal{X}(0)=x>0$ a.s.\
					and $m_x(y,t)$ be the marginal probability density function of $\mathcal{X}(t)$, for all $t\geq0$.
					Then $m_x(y,t)$ satisfies the equation
					\begin{equation*}
						\partial^{\beta}_t m_x(y,t)=\frac{\partial }{\partial y} \left(bym_x(y,t)\right)
						+\frac{\partial^2}{\partial y^2}\left( cym_x(y,t)\right),
					\end{equation*}
					where $\partial^{\beta}_t$ is the D\v{z}rba\v{s}jan--Caputo derivative.
				\end{prop}
				Moreover, note that it is possible to write explicitely the SDE solved by the process $\{\mathcal{X}(t)\}_{t\geq0}$.
				Let $\left(\Omega, \mathcal{F},\mathbb{G}=\{\mathcal{G}_t\}_{t\geq0},\mathbf{P}\right)$ be a filtered probability space
				and let $D=\{D(t)\}_{t\geq0}$ be a $\mathbb{G}$-adapted stable subordinator of parameter $\beta\in(0,1)$.
				Furthermore, let $\mathbb{F}=\{\mathcal{F}_t\}_{t\geq0}$
				be the filtration defined by means of time-change with the process $\{\mathcal{E}(t)\}_{t\geq0}$,
				inverse of $D$, such that, for all $t\geq0$, $\mathcal{F}_t=\mathcal{G}_{\mathcal{E}(t)}$
				(see \cite{MR542115}, page 312).
				Consider the filtered space $\left(\Omega, \mathcal{F},\mathbb{F},\mathbf{P}\right)$
				and suppose $\{X(t)\}_{t\geq0}$ is a $\mathbb{G}$-adapted
				Feller branching diffusion. Then the process $\{\mathcal{X}(t)\}_{t\geq0}$ is solution of the SDE
				\begin{equation*}
					d\mathcal{X}(t)=-b\mathcal{X}(t)d\mathcal{E}(t)+\sqrt{2c\mathcal{X}(t)}dW_{\mathcal{E}(t)},
				\end{equation*}
				where $\{W_{\mathcal{E}(t)}\}_{t\geq0}$ is an $\mathbb{F}$-adapted time-changed Brownian motion, 
				also known as grey Brownian motion.

			\subsubsection{Time-changed Yule process} 

				Let us consider a homogeneous Poisson process $\{Y(t)\}_{t\geq0}$ with rate $\theta>0$
				and shifted upwards by 1.
				By relation \eqref{exponent}, it is transformed into a CSBP $\{X(t)\}_{t\geq0}$ with Laplace exponent
				\begin{align*}
					\nu_t(\lambda)=\log(1-(1-e^{\lambda})e^{\theta t}), \qquad t\geq0 \: \lambda\geq 0.
				\end{align*}
				This is the Laplace exponent of a Yule process $\{X(t)\}_{t\geq0}$, that is a pure birth process with linear birth rate.
				If $X(0)$ is supported on the strictly positive integers, then the law of $X$ at every time $t\geq0$ is a probability measure
				$\{p(\cdot,t)\}$ satisfying
				\begin{align*}
					\frac{\partial}{\partial t} p(n,t)=\theta (n-1) p(n-1,t)-\theta n p(n,t), \qquad n \ge 1.
				\end{align*}
				The time-changed Yule process $\{\mathcal{X}(t)\}_{t\geq0}$ is studied in \cite{OrPo10}, where
				amongst other properties it is proved that for each $t\geq0$ its law is a probability measure
				$p_{\beta}(\cdot, t)$ that satisfies the time-fractional difference-differential equations
				\begin{align*}
					\partial^{\beta}_t p_{\beta}(n,t)=\theta (n-1) p_{\beta}(n-1,t)-\theta n p_{\beta}(n,t), \qquad n \ge 1,
				\end{align*}
				and whose explicit form is
				\begin{align*}
					p_{\beta}(n,t)=\sum_{j=1}^n\binom{n-1}{j-1} (-1)^{j-1}E_{\beta}(-\theta j t^{\beta}), \qquad n \ge 1,
				\end{align*}
				that is consistent with our results in Section~\ref{momenti_sec}.
				
	\subsubsection*{Acknowledgements} 
	
		F.\ Polito and L.\ Sacerdote have been supported by the projects \emph{Memory in Evolving Graphs}
		(Compagnia di San Paolo/Universit\`a di Torino) and by INDAM (GNAMPA/GNCS). F.\ Polito has also been supported
		by project \emph{Sviluppo e analisi di processi Markoviani e non Markoviani con applicazioni} (Universit\`a di Torino).
		L.\ Andreis has been partially supported by Centro Studi Levi Cases (Universit\`a di Padova).

	\bibliographystyle{abbrv} 
	\bibliography{paper}

\end{document}